%% file: 0-compile.tex
\documentclass{article}%
\usepackage{amsmath}
\usepackage{amsfonts}
\usepackage{amssymb}
\usepackage{graphicx}
\usepackage{hyperref}%
\newtheorem{theorem}{Theorem}

\newtheorem{corollary}[theorem]{Corollary}

\newtheorem{definition}[theorem]{Definition}

\newtheorem{lemma}[theorem]{Lemma}

\newtheorem{remark}[theorem]{Remark}

\newenvironment{proof}[1][Proof]{\noindent\textbf{#1.} }{\ \rule{0.5em}{0.5em}}
\title{Computer assisted proof of Shil'nikov homoclinics: with application to the Lorenz-84 model}

\author{Maciej J. Capi\'nski \thanks{ Faculty of Applied Mathematics, AGH University of Science and Technology, al. Mickiewicza 30, 30-059
Krak\'ow,  Poland ({\tt maciej.capinski@agh.edu.pl})  supported by NCN Grant ***** } \and 
Anna Wasieczko-Zajac \thanks{Faculty of Applied Mathematics, AGH University of Science and Technology, al. Mickiewicza 30, 30-059
Krak\'ow,  Poland ({\tt wasieczk@agh.edu.pl})}}

\usepackage[polish,english]{babel}
\usepackage{polski}

\newcommand{\TheTitle}{Computer assisted proof of Shil'nikov homoclinics} 
\newcommand{\TheAuthors}{M. J. Capi\'nski, and A. Wasieczko-Zaj\k{a}c}

\title{Computer assisted proof of Shil'nikov homoclinics: with application to the Lorenz-84 model}

\author{Maciej J. Capi\'nski \thanks{ Faculty of Applied Mathematics, AGH University of Science and Technology, al. Mickiewicza 30, 30-059
Krak\'ow,  Poland ({\tt maciej.capinski@agh.edu.pl}) } \and 
Anna Wasieczko-Zaj\k{a}c \thanks{Faculty of Applied Mathematics, AGH University of Science and Technology, al. Mickiewicza 30, 30-059
Krak\'ow,  Poland ({\tt wasieczk@agh.edu.pl})}}

\usepackage{amsopn}

\ifpdf
\hypersetup{
  pdftitle={\TheTitle},
  pdfauthor={\TheAuthors}
}
\fi

\begin{document}

\maketitle

\begin{abstract}
  We present a methodology for computer assisted proofs of Shil'nikov homoclinic intersections. It is based on geometric bounds on the invariant manifolds using rate conditions, and on propagating the bounds by an interval arithmetic integrator. Our method ensures uniqueness of the parameter for which the homoclinic takes place. We apply the method for the Lorenz-84 atmospheric circulation model, obtaining a sharp bound for the parameter, and also for where the homoclinic intersection of the stable/unstable manifolds takes place.
\end{abstract}


\input{Shiln-Lor84}

\input{ref}
\end{document}

%% file: Shiln-Lor84.tex

\section{Introduction}

A class of three dimensional systems with a homoclinic orbit for a three
dimensional saddle-focus equilibrium point was studied by Shil'nikov in a
series of papers (see for example \cite{ShilBif1}, \cite{ShilBif2}, \cite%
{ShilBif3}). The homoclinic (usually called the Shil'nikov homoclinic
orbit), can bifurcate in simple as well as in a chaotic way. The type of
bifurcation depends on the saddle quantity, a constant derived from the
eigenvalues of the linearised vector field at the fixed point. If the saddle
quantity is negative, then a unique and stable limit cycle bifurcates from
the homoclinic orbit. (This is called the simple Shil'nikov bifurcation.) If
it is negative, then there occurs infinitely many periodic orbits of saddle
type and one speaks of the chaotic Shil'nikov bifurcation (see also \cite%
{kuznetsov}). Shil'nikov homoclinics are important, since they lead to
interesting dynamics. For instance, when combined with the study of the
separatrix value, once can infer from them the existence of a Lorenz type
attractor in the system \cite{Turaev}.

Detecting Shil'nikov homoclinic intersections analytically is difficult,
since in most systems of interest the ODE does not have a closed-form
solution. In this paper we present a computer assisted approach for such
proofs. The method is based on computer assisted estimates on the stable and
unstable manifolds, and their propagation using rigorous, interval
arithmetic integrator along the flow.

Our estimates for the invariant manifolds are based on the method of `rate
conditions' from \cite{CZ-Meln, CZ-nhims}. These are related to the rate
conditions of Fenichel \cite{Fen1,Fen2,Fen3,Fen4}. The difference is that
our our rate conditions are derived based on the estimates on the derivative
at a (large) neighbourhood of a normally hyperbolic manifold (in this paper
this manifold will be a family of hyperbolic fixed points), and not at the
manifold as is done by Fenichel. Since our estimates are more global, we are
able to establish existence and obtain explicit bounds on the invariant
manifolds within the investigated neighbourhood.

The bounds on the manifolds are then propagated along the flow using
interval and arithmetic integrator. For the proof of a homoclinic
intersection, we use a standard shooting argument, which is based on the
Bolzano's intermediate value theorem. We also keep track of the dependence
of the manifolds on the parameter, which leads to a uniqueness argument for
the intersection.

To demonstrate that our method is applicable we implement it for the
Lorenz-84 system \cite{Lorenz}. We make a list of conditions that need to be
verified in order to obtain the existence and uniqueness of the
intersection, and then validate them. The bounds obtained by us are quite
sharp. We establish the intersection parameter with $10^{-9}$ order of
accuracy, and the region where the intersection takes place with $10^{-7}$
order of accuracy. The Lorenz-84 model serves only as an example. Our method
is general, and can be applied to other systems.

The only other computer assisted proof of Shil'nikov homoclinics known to us
is the work of Wilczak \cite{Wilczak}. This method uses a topological
shadowing mechanism, which stems from the method of `covering relations' 
\cite{GZ1, GZ2} (refered to also in literature as `correctly aligned
windows'), and Lyapunov function type arguments close to the fixed points.
Our method is different. We rely on explicit estimates on the manifolds and
their slopes, which are derived from rate conditions. Our method implies
that the intersection parameter is unique within the given range. The
uniqueness was not investigated in \cite{Wilczak}. In \cite{Wilczak} it is
shown that in the investigated system there is an infinite number of
Shil'nikov homoclinics, that are derived from symbolic dynamics. We focus on
a simpler setting where the intersection is unique.

The paper is organised as follows. Section \ref{sec:prel} contains
preliminaries. In section \ref{sec:L84} we introduce the Lorenz-84 model.
Section \ref{sec:Shiln} contains the proof for Shil'nikov type bifurcations.
The proof is based on an assumption that within the investigated
neighbourhood of the family of hyperbolic fixed points we have estimates on
their invariant manifolds. We discuss how to obtain such estimates in
section \ref{sec:Wu-bound}. This is based on the `rate conditions' method
from \cite{CZ-nhims, CZ-Meln}, adapted to our setting. In section \ref%
{sec:Wu-param} we extend the method to obtain bounds on the dependence of
the manifolds on the parameter of the system. Finally, in section \ref%
{sec:CAP}, we apply our method for the Lorenz-84 system.

\section{Preliminaries\label{sec:prel}}

\subsection{Notations}

Throughout the paper, all norms that appear are standard Euclidean norms. We
use a notation $B_k(p,R)$ to denote a ball in $\mathbb{R}^k$ of radius $R$
centered at $p$. We use a short hand notation $B_{k}\left( R\right) $ for a
ball or radius $R$ in $\mathbb{R}^{k}$ centered at zero. For a set $A\subset%
\mathbb{R}^{k}$, we use $\overline{A}$ to denote its closure and $\partial A$
for its boundary, $\mathrm{int}A$ for its interior and $A^{c}$ for the
complement. For a point $p=\left( x,y\right) $ we use a notation $\pi_{x}p$
and $\pi_{y}p$ to denote projections onto $x$ and $y$ coordinates,
respectively. We use the notation $(v|w)$ for the scalar product between two
vectors $v$ and $w$.

\subsection{Interval Newton method}

Let $X$ be a subset of $\mathbb{R}^{n}$. We shall denote by $[X]$ an
interval enclosure of the set $X$, that is, a set 
\begin{equation*}
\lbrack X]=\Pi_{i=1}^{n}[a_{i},b_{i}]\subset\mathbb{R}^{n},
\end{equation*}
such that 
\begin{equation*}
X\subset\lbrack X].
\end{equation*}

Let $f:\mathbb{R}^{n}\to\mathbb{R}^{n}$ be a $C^{1}$ function and $U\subset%
\mathbb{R}^{n}$. We shall denote by $[Df(U)]$ the interval enclosure of a
Jacobian matrix on the set $U$. This means that $[Df(U)]$ is an interval
matrix defined as 
\begin{equation*}
\lbrack Df(U)]=\left\{ A\in\mathbb{R}^{n\times n}|A_{ij}\in\left[ \inf_{x\in
U}\frac{df_{i}}{dx_{j}}(x),\sup_{x\in U}\frac{df_{i}}{dx_{j}}(x)\right] 
\text{ for all }i,j=1,\ldots,n\text{ }\right\} .
\end{equation*}

\begin{theorem}
\cite{Al}\label{th:Newton} (Interval Newton method) \label%
{th:interval-Newton} Let $f:\mathbb{R}^{n}\rightarrow\mathbb{R}^{n}$ be a $%
C^{1}$ function and $X=\Pi_{i=1}^{n}[a_{i},b_{i}]$ with $a_{i}<b_{i}$. If $%
[Df(X)]$ is invertible and there exists an $x_{0}$ in $X$ such that%
\begin{equation*}
N(x_{0},X):=x_{0}-\left[ Df(X)\right] ^{-1}f(x_{0})\subset X,
\end{equation*}
then there exists a unique point $x^{\ast}\in X$ such that $f(x^{\ast})=0.$
\end{theorem}

The Interval Newton Method can be applied to find the eigenvalues and
eigenvectors of a matrix.

\subsection{Interval arithmetic enclosure for eigenvalues and eigenvectors 
\label{sec:eigenvalues}}

Let $A$ be an $n\times n$ real matrix. In this section we outline how to
solve%
\begin{equation}
Ax=\lambda x.  \label{eq:eigenvalue-equation}
\end{equation}
We consider two cases. In the first, both $\lambda$ and $x$ will be real,
and in the second 
\begin{align*}
\lambda & =\rho+i\omega, \\
x & =x_{\mathrm{re}}+ix_{\mathrm{im}},
\end{align*}
will be complex.

In the first case, we fix the first coordinate $x_{1}$ of $x=\left( x_{1},%
\tilde{x}\right) $ and treat $\tilde{x}\in\mathbb{R}^{n-1}$ as a variable.
(We can also set some other coordinate to be fixed, if needed.) We define $f:%
\mathbb{R}^{n}\rightarrow\mathbb{R}^{n}$ as%
\begin{equation*}
f\left( \lambda,\tilde{x}\right) =Ax-\lambda x.
\end{equation*}
We see that solving $f\left( \lambda,\tilde{x}\right) =0$ is equivalent to (%
\ref{eq:eigenvalue-equation}). A solution of $f\left( \lambda,\tilde {x}%
\right) =0$ can be established using the interval Newton method (Theorem \ref%
{th:Newton}).

In the second case, we can consider $x_{\mathrm{re}}=\left( x_{\mathrm{re}%
,1},\tilde{x}_{\mathrm{re}}\right) $ and $x_{\mathrm{im}}=\left( x_{\mathrm{%
im},1},\tilde{x}_{\mathrm{im}}\right) ,$ treating $\tilde {x}_{\mathrm{re}},%
\tilde{x}_{\mathrm{im}}\in\mathbb{R}^{n-1}$ as variables and $x_{\mathrm{re}%
,1},x_{\mathrm{im},1}$ as fixed parameters. (We can also fix some other
coordinate than the first, if needed.) We can consider $f:\mathbb{R}%
^{2n}\rightarrow\mathbb{R}^{2n}$ defined as%
\begin{equation*}
f\left( \rho,\tilde{x}_{\mathrm{re}},\omega,\tilde{x}_{\mathrm{im}}\right)
=\left( 
\begin{array}{c}
Ax_{\mathrm{re}}-\rho x_{\mathrm{re}}+\omega x_{\mathrm{im}} \\ 
Ax_{\mathrm{im}}-\rho x_{\mathrm{im}}-\omega x_{\mathrm{re}}%
\end{array}
\right) .
\end{equation*}
Clearly $f\left( \rho,\tilde{x}_{\mathrm{re}},\omega,\tilde{x}_{\mathrm{im}%
}\right) =0$ is equivalent to (\ref{eq:eigenvalue-equation}), and the
solution can again be established using the interval Newton method.

\subsection{Linear approximation of solutions of ODEs}

In this section we present a technical lemma. Consider an ODE 
\begin{equation*}
p^{\prime}=f(p),
\end{equation*}
where $f:\mathbb{R}^{n}\rightarrow\mathbb{R}^{n}$ is $C^{1}$. Let $\Phi_{t}$
be the flow of the above system.

\begin{lemma}
\label{lem:phi(p)-phi(q)} Let $U\subset\mathbb{R}^{n}$ be a convex compact
set. Then there exsist a constant $M>0$ such that for any $t>0$ and any $%
p,q\in\mathbb{R}^{n}$ satisfying 
\begin{equation*}
\{\Phi_{-s}(p),\Phi_{-s}(q):s\in\lbrack 0 ,t]\}\subset U,
\end{equation*}
we have 
\begin{equation*}
\Phi_{-t}(p)-\Phi_{-t}(q)=p-q-tC(p-q)+g(t,p,q),
\end{equation*}
for some matrix $C\in\left[ Df\left( U\right) \right] $ (which can depend on 
$p$, $q$ and $t$)%
and some $g$ satisfying 
\begin{equation*}
\left\Vert g(t,p,q)\right\Vert \leq Mt^{2}\left\Vert p-q\right\Vert .
\end{equation*}
\end{lemma}

\begin{proof}
The proof is given in the Appendix.
\end{proof}

\begin{remark} In Lemma \ref{lem:phi(p)-phi(q)} we move backwards in time along the flow. We set this up in this way, because later on in our application we will use the lemma in the context of unstable manifolds, where moving back in time along the flow we will converge towards a fixed point.
\end{remark}

\subsection{Logarithmic norms}

Let us begin with defining some matrix functionals that will be used by us
in further proofs. Let $\left\Vert \cdot\right\Vert $ be a given norm in $%
\mathbb{R}^{n}$. Let $A\in\mathbb{R}^{n\times n}$ be a square matrix. By $%
m(A)$ we will denote the following matrix functional: 
\begin{equation*}
m(A)=\min_{z\in\mathbb{R}^{n},\Vert z\Vert=1}\Vert Az\Vert.
\end{equation*}

\begin{definition}
\emph{The logarithmic norm of} $A$, denoted by $l(A)$ by \cite{L,D,HNW,KZ},
is defined as 
\begin{equation}
l(A)=\lim_{h\rightarrow0^{+}}\frac{\Vert I+hA\Vert-\Vert I\Vert}{h}.
\label{eq:def-log-norm}
\end{equation}
Moreover 
\begin{equation}
m_{l}(A)=\lim_{h\rightarrow0^{+}}\frac{m(I+hA)-\Vert I\Vert}{h}
\label{eq:def-ml}
\end{equation}
will be called \emph{the logarithmic minimum of} $A$.
\end{definition}

\begin{lemma}
\label{lem:norms-spectrum} If $\left\Vert \cdot\right\Vert $ is the
Euclidean norm, then the following equalities hold 
\begin{align}
l(A) & =\max\{\lambda\in\text{spectrum of }(A+A^{\top})/2\},
\label{eq:eucl-log-norm} \\
m_{l}(A) & =\min\{\lambda\in\text{spectrum of }(A+A^{\top})/2\}.
\label{eq:eucl-ml}
\end{align}
\end{lemma}

\begin{remark}
Equality (\ref{eq:eucl-log-norm}) is a well known result (see for instance 
\cite{HNW}). Equation (\ref{eq:eucl-ml}) is proven in \cite{CZ-Meln}.
\end{remark}

\begin{corollary}
\label{cor:m-to-l}From Lemma \ref{lem:norms-spectrum}, we see that $%
m_{l}(-A)=-l\left( A\right) .$
\end{corollary}

\begin{lemma}
\cite{CZ-Meln} \label{lem:norn-l-approx}Consider the Euclidean norm $%
\left\Vert \cdot\right\Vert $. Let $W\subset\mathbb{R}^{n\times n}$ be a
compact set and let $t_{0}>0$. Then for any $t\in(0,t_{0}]$ and $A\in W$ the
following equality holds 
\begin{equation*}
\Vert I+tA\Vert=1+tl(A)+r(t,A),
\end{equation*}
where 
\begin{equation*}
\Vert r(t,A)\Vert\leq Ct^{2},
\end{equation*}
for some constant $C=C(t_{0},W)$.
\end{lemma}

\begin{lemma}
\label{lem:m-to-ml}\cite{CZ-Meln} Consider the Euclidean norm $\left\Vert
\cdot\right\Vert $. Let $W\subset\mathbb{R}^{n\times n}$ be a compact set
and let $t_{0}>0$. Then for any $t\in(0,t_{0}]$ and $A\in W$ the following
equality holds 
\begin{equation*}
m(I+tA)=1+tm_{l}(A)+r(t,A)
\end{equation*}
where 
\begin{equation*}
\Vert r(t,A)\Vert\leq Ct^{2}
\end{equation*}
for some constant $C=C(t_{0},W)$.
\end{lemma}

\begin{figure}[ptb]
\begin{center}
\includegraphics[height=5cm]{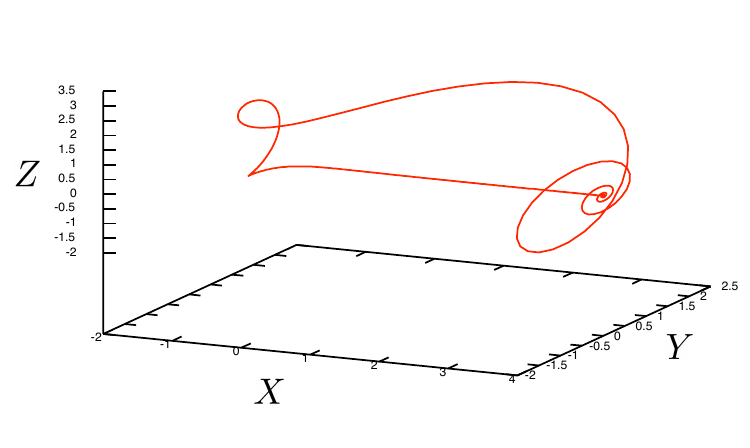}
\end{center}
\caption{{}The Shil'nikov homoclinic in the Lorenz-84 model for $F=4.0$ and $%
G\simeq0.08$.}
\label{fig:homoclinic}
\end{figure}

\section{Lorenz-84 Atmospheric Circulation Model\label{sec:L84}}

The Lorenz-84 Model was introduced by Lorenz in \cite{Lorenz}. It is a
low-order model for the long-term atmospheric circulation. It is considered
as the simplest model capable of representing the basic features of the
so-called Hadley circulation. Therefore, it has been widely used in
meteorogical studies. The detailed analysis of this model can be found in 
\cite{Veen}. The model equations are%
\begin{equation}  \label{eq:L84}
\left\{ 
\begin{array}{rl}
\dot{X} & =-Y^{2}-Z^{2}-aX+aF, \\ 
\dot{Y} & =XY-bXZ-Y+G, \\ 
\dot{Z} & =bXY+XZ-Z,%
\end{array}
\right.
\end{equation}
where variable $X$ represents the strength of the globally averaged westerly
wind current and variables $Y$ and $Z$ are the strength of the cosine and
sine phases of a chain of superposed waves transporting heat poleward. $F$
and $G$ represent the thermal forcing terms, and the parameter $b$ stands
for the advection strength of the waves by the westerly wind current. The
coefficient $a$, if less than 1, allows the westerly wind current to damp
less rapidly than the waves. The time unit is equal to the damping time of
the waves and it is estimated to be five days.

In their paper \cite{Shil}, A.Shil'nikov, G.Nicolis and C.Nicolis carry out
a detailed bifurcation analysis for the Lorenz-84 Model with parameters $a$
and $b$ set to classical values $\frac{1}{4}$ and $4$ respectively (these
values were also considered in many other works, see for example \cite%
{broer-simo}, \cite{Lorenz}, \cite{Lorenz2}). The authors identify the types
of the equilibrium points depending on the choice of the domain for the
parameters $F$ and $G$. They show that the problem has either one, two or
three equilibrium points. If parameters $F$ and $G$ are chosen from a proper
domain, one of the fixed points, denoted in \cite{Shil} as $O_{1}$, is
saddle-focus. The paper \cite{Shil} presents a numerical calculations
suggesting the existence of the homoclinic orbit passing through $O_{1}$
that is possesed by the system for $F\simeq4.0$ and $G\simeq0.08$. The
homoclinic is depicted in Figure \ref{fig:homoclinic}.

Following Shil'nikov et al. \cite{Shil} we set parameters $a=\frac{1}{4}$
and $b=4$. In further sections we will use our method to rigourously enclose
the stable and unstable manifolds, and to validate the existence of a
homoclinic orbit for saddle-focus fixed point $O_{1}$. We prove that such an
orbit exists for $F=4$, and some $G$, where 
\begin{equation}
G\in\left[ 0.0752761095,0.07527611625 \right] .  \label{eq:G-estim}
\end{equation}
Moreover, we show the uniqueness of such $G$ in the interval (\ref%
{eq:G-estim}).

\section{Establishing Shil'nikov homoclinics\label{sec:Shiln}}

Let us consider the three dimensional system given by the following ODE 
\begin{equation}
p^{\prime}=f(p,\theta),  \label{eq:ode_Shil}
\end{equation}
where $f:\mathbb{R}^{3}\times\mathbb{R}\rightarrow\mathbb{R}^{3}$ is $C^{1}$%
, and $\theta\in\Theta$ is a parameter, with $\Theta=[\theta_{l},\theta
_{r}]\subset\mathbb{R}$. Let $\Phi_{t}(p,\theta)$ be the flow induced by (%
\ref{eq:ode_Shil}).

Suppose that for $\theta\in\Theta$, system \ref{eq:ode_Shil} has a smooth
family of hyperbolic fixed points $p_{\theta}^{\ast}$, with two dimensional
stable and one dimensional unstable eigenspace. 
\begin{figure}[ptb]
\begin{center}
\includegraphics[height=3cm]{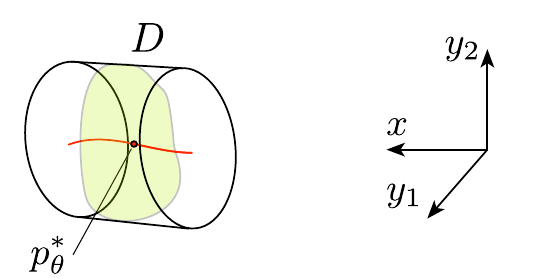}
\end{center}
\caption{The local unstable manifold $W_{\protect\theta}^{u}$ in red, and
the local stable manifold $W_{\protect\theta}^{s}$ in green.}
\label{fig:D-setup}
\end{figure}

Below we present a theorem, which allow us to prove the existence of a
homoclinic orbit in the system. First we need to introduce some notation.

Let $\overline{B}_{u}\left( R\right) =\left[ -R,R\right] \subset \mathbb{R}$%
, $\overline{B}_{s}\left( R\right) \subset\mathbb{R}^{2}$ and let%
\begin{equation*}
D=\overline{B}_{u}\left( R\right) \times\overline{B}_{s}\left( R\right)
\subset\mathbb{R}^{3},
\end{equation*}
be a neighborhood of the smooth family of fixed points, meaning
that we assume $p_{\theta}^{\ast}\in\mathrm{int}D$ for any $\theta\in\Theta$%
. The set $D$ will be fixed throughout the discussion. We denote by $%
W_{\theta}^{u}$ the local unstable manifold of $p_{\theta}^{\ast}$
in $D$ and by $W_{\theta}^{s}$ the local stable manifold of $%
p_{\theta}^{\ast}$ in $D$, i.e.%
\begin{align}
W_{\theta}^{u} & =\left\{ p\in D:\Phi_{t}\left( p,\theta\right) \in D\text{
for }t\leq0\text{ and }\lim_{t\rightarrow-\infty}\Phi_{t}\left(
p,\theta\right) =p_{\theta}^{\ast}\right\} ,  \label{eq:Wu-def} \\
W_{\theta}^{s} & =\left\{ p\in D:\Phi_{t}\left( p,\theta\right) \in D\text{
for }t\geq0\text{ and }\lim_{t\rightarrow+\infty}\Phi_{t}\left(
p,\theta\right) =p_{\theta}^{\ast}\right\} .  \label{eq:Ws-def}
\end{align}
We assume that $W_{\theta}^{u}$ and $W_{\theta}^{s}$ are graphs of $C^{1}$
functions%
\begin{align*}
w^{u} & :\overline{B}_{u}\left( R\right) \times\Theta\rightarrow \overline{B}%
_{s}\left( R\right) , \\
w^{s} & :\overline{B}_{s}\left( R\right) \times\Theta\rightarrow \overline{B}%
_{u}\left( R\right) ,
\end{align*}
meaning that (see Figure \ref{fig:D-setup})%
\begin{align}
W_{\theta}^{u} & =\left\{ \left( x,w^{u}\left( x,\theta\right) \right) :x\in%
\overline{B}_{u}\left( R\right) \right\} ,  \notag \\
W_{\theta}^{s} & =\left\{ \left( w^{s}\left( y,\theta\right) ,y\right) :y\in%
\overline{B}_{s}\left( R\right) \right\} .  \label{eq:Ws-graph}
\end{align}

\begin{figure}[ptb]
\begin{center}
\includegraphics[height=5cm]{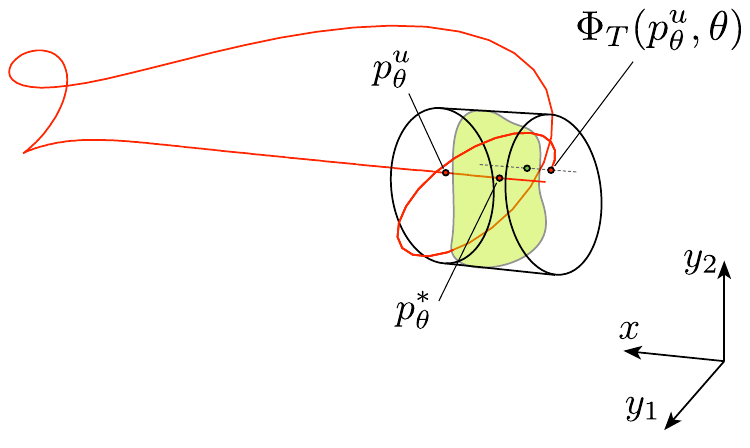}
\end{center}
\caption{{}We have the $1$-dimensional unstable manifold of $p_{\protect%
\theta}^{\ast }$ in red, and the $2$-dimensional local stable manifold $W_{%
\protect\theta}^{s}$ in $D$ in green. The $h\left( \protect\theta\right) $
is the signed distance along the $x$ coordinate between $W_{\protect\theta%
}^{s}$ and $\Phi_{T}\left( p_{\protect\theta}^{u},\protect\theta\right) $;
this is the distance along the dotted line on the plot. }
\label{fig:homoclinic-2}
\end{figure}

Let 
\begin{equation}
p_{\theta}^{u}:=\left( R,w^{u}\left( R,\theta\right) \right) \in \mathbb{R}%
^{3}.  \label{eq:pu-theta-def}
\end{equation}
Consider $T>0$ and assume that for all $\theta\in\Theta$, $\Phi_{T}\left(
p_{\theta}^{u},\theta\right) \in D$. Let us define%
\begin{equation*}
h:\Theta\rightarrow\mathbb{R},
\end{equation*}
as%
\begin{equation}
h\left( \theta\right) =\pi_{x}\Phi_{T}\left( p_{\theta}^{u},\theta\right)
-w_{\theta}^{s}(\pi_{y}\Phi_{T}\left( p_{\theta}^{u},\theta\right) ).
\label{eq:h-function-def}
\end{equation}

We now state a natural result, that $h\left( \theta\right) =0$ implies an
intersection of the stable and unstable manifolds of $p_{\theta}^{\ast}$.
(See Figure \ref{fig:homoclinic-2}.)

\begin{theorem}
\label{th:homoclinic-existence}If 
\begin{equation}
h(\theta_{l})<0\qquad\text{and}\qquad h(\theta_{r})>0
\label{eq:Bolzano-assmpt}
\end{equation}
then there exists a $\psi\in\Theta$ for which we have a homoclinic orbit to $%
p_{\psi}^{\ast}$.

Moreover, if for all $\theta\in\Theta$, $h^{\prime}(\theta)>0$, then $\psi$
is the only parameter for which we have a homoclinic orbit satisfying $%
\Phi_{t}\left( p_{\theta}^{u},\theta\right) \in D$ for all $t>T$.
\end{theorem}

\begin{proof}
Since $w^{u}$, $w^{s}$ are $C^{1}$, also is $h$. From (\ref%
{eq:Bolzano-assmpt}), by the Bolzano intermediate value theorem, it follows
that there exists a $\psi \in \Theta $ for which $h\left( \psi \right) =0$.
Let $q=\Phi _{T}(p_{\psi }^{u},\psi )$. Since $h\left( \psi \right) =0,$ 
\begin{equation}
q=\left( \pi _{x}q,\pi _{y}q\right) =\left( w_{\psi }^{s}(\pi _{y}q),\pi
_{y}q\right) .  \label{eq:q-on-manif}
\end{equation}%
Since $p_{\psi }^{u}\in W^{u},$ clearly $q=\Phi _{T}( p_{\psi }^{u},\psi ) $
belongs to the unstable manifold of $p_{\psi }^{\ast }$. All points of the
form $( w_{\psi }^{s}(y),y) $ belong to the stable manifold of $p_{\psi
}^{\ast }$, hence by (\ref{eq:q-on-manif}) so does $q$, and the stable and
unstable manifolds intersect at $q$.

If $h^{\prime }\left( \theta \right) >0$ for all $\theta \in \Theta $, then $%
\psi $ is the only parameter for which $h$ is zero, hence for all $\theta
\neq \psi $, 
\begin{equation*}
\pi _{x}\Phi _{T}\left( p_{\theta }^{u},\theta \right) \neq w_{\theta
}^{s}(\pi _{y}\Phi _{T}\left( p_{\theta }^{u},\theta \right) ).
\end{equation*}%
This by (\ref{eq:Ws-graph}), implies that for $\theta \neq \psi $, $\Phi
_{T}\left( p_{\theta }^{u},\theta \right) \notin W_{\theta }^{s}$. By (\ref%
{eq:Ws-def}) this means that for some $t>T$, $\Phi _{t}\left( p_{\theta
}^{u},\theta \right) \notin D$, or that we do not have a homoclinic for this
parameter.
\end{proof}

\begin{remark}
The inequalities in (\ref{eq:Bolzano-assmpt}) and the sign of $h^{\prime}$
in Theorem \ref{th:homoclinic-existence} can be reversed. Then the proof
follows from mirror arguments.
\end{remark}

To apply Theorem \ref{th:homoclinic-existence}, we need to be able to
compute estimates on $h$ and its derivative. We note that obtaining a
rigorous bound on a time shift map $\Phi_{T}$ along the flow, and on its
derivative, can be computed in interval arithmetic using the CAPD\footnote{%
computer assisted proofs in dynamics: http://capd.ii.uj.edu.pl/} package. To
compute $h$ and its derivative it is therefore enough to be able to obtain
estimates on $w^{u},$ $w^{s}$ and their derivatives. We discuss how this can
be achieved in interval arithmetic in subsequent sections \ref{sec:Wu-bound}
and \ref{sec:Wu-param}. We use these, together with Theorem \ref%
{th:homoclinic-existence}, to provide a computer assisted proof of a
homoclinic intersection in the Lorenz-84 model, in section \ref{sec:CAP}.

\section{Bounds on unstable manifolds of hyperbolic fixed points\label%
{sec:Wu-bound}}

Consider an ODE%
\begin{equation}
q^{\prime}=f\left( q\right) ,  \label{eq:ode-no-param}
\end{equation}
and let%
\begin{equation*}
D=\overline{B}_{u}\left( R\right) \times\overline{B}_{s}\left( R\right)
\subset\mathbb{R}^{u}\times\mathbb{R}^{s}
\end{equation*}

The results of this section are more general than the previously considered
ode in $\mathbb{R}^{3}$, and here $u,s$ can be any natural numbers.
We use a notation $x\in\mathbb{R}^{u}$ to stand for the unstable coordinate
and $y\in \mathbb{R}^{s}$ for the stable coordinate. For us it will be
enough if these coordinates are `roughly' aligned with the eigenspaces of a
fixed point. (We do not need to work with precisely linearised local
coordinates.) We write $f(x,y)=(f_x(x,y),f_y(x,y))$, where $f_x$ is the
projection onto $\mathbb{R}^u $ and $f_y$ is the projection onto $\mathbb{R}%
^s$.

Let $L>0$ be a fixed number. We define%
\begin{align*}
\mu_{1} & =\sup_{z\in D}\left\{ l\left( \frac{\partial f_{y}}{\partial y}%
(z)\right) +\frac{1}{L}\left\Vert \frac{\partial f_{y}}{\partial x}%
(z)\right\Vert \right\} , \\
\mu_{2} & =\sup_{z\in D}\left\{ l\left( \frac{\partial f_{y}}{\partial y}%
(z)\right) +\frac{1}{L}\left\Vert \frac{\partial f_{x}}{\partial y}%
(z)\right\Vert \right\} , \\
\xi & =m_{l}\left( \frac{\partial f_{x}}{\partial x}(D)\right) -\frac{1}{L}%
\sup_{z\in D}\left\Vert \frac{\partial f_{x}}{\partial y}(z)\right\Vert .
\end{align*}

\begin{definition}
We say that the vector field $f$ \emph{satisfies rate conditions in $D$} if%
\begin{equation}
\mu_{1}<0<\xi,  \label{eq:rate-cond1}
\end{equation}%
\begin{equation}
\mu_{2}<\xi.  \label{eq:rate-cond2}
\end{equation}
\end{definition}

\begin{definition}
We say that $D=\overline{B}_{u}\left( R\right) \times\overline{B}_{s}\left(
R\right) $ is \emph{an isolating block for (\ref{eq:ode-no-param})} if

\begin{enumerate}
\item For any $q\in\partial\overline{B}_{u}\left( R\right) \times \overline{B%
}_{s}\left( R\right) $, 
\begin{equation*}
\left( \pi_{x}f(q)|\pi_{x}q\right) >0.
\end{equation*}

\item For any $q\in\overline{B}_{u}\left( R\right) \times\partial \overline{B%
}_{s}\left( R\right) $,%
\begin{equation*}
\left( \pi_{y}f(q)|\pi_{y}q\right) <0.
\end{equation*}
\end{enumerate}
\end{definition}

\begin{definition}
We define the unstable set in $D$ as%
\begin{equation*}
W^{u}=\{z:\text{ }\Phi_{t}(z)\in D\text{ for all }t<0\}.
\end{equation*}
\end{definition}

\begin{theorem}
\label{th:Wu-rates} Assume that $f$ is $C^{1}$ and satisfies rate
conditions. Assume also that $D=\overline{B}_{u}\left( R\right) \times%
\overline{B}_{s}\left( R\right) $ is an isolating block for $f$.
Then the set $W^{u}$ is a manifold, which is a graph over $\overline{B}%
_{u}\left( R\right) $. To be more precise, there exists a function 
\begin{equation*}
w^{u}:\overline{B}_{u}(R)\rightarrow\overline{B}_{s}(R),
\end{equation*}
such that 
\begin{equation*}
W^{u}=\left\{ \left( x,w^{u}(x)\right) :x\in\overline{B}_{u}(R)\right\} .
\end{equation*}
Moreover, $w^{u}$ is Lipschitz with constant $L$ and for $C=2R\left(
1+1/L\right) $, for any $p_{1},p_{2}\in W^{u},$%
\begin{equation}
\left\Vert \Phi_{-t}\left( p_{1}\right) -\Phi_{-t}\left( p_{2}\right)
\right\Vert \leq Ce^{-t\xi}\qquad\text{for all }t>0.
\label{eq:contraction-cond}
\end{equation}
\end{theorem}

\begin{proof}
The result follows directly from Theorem 30 from \cite{CZ-Meln}. Theorem 30
in \cite{CZ-Meln} is written in the context where apart from $x,y$ we have
an additional `center' coordinate, which is not present here. This is why
the number of constants and rate conditions (\ref{eq:rate-cond1}--\ref%
{eq:rate-cond2}) for Theorem \ref{th:Wu-rates} is smaller than the number of
constants and associated inequalities needed in \cite{CZ-Meln}. The (\ref%
{eq:rate-cond1}--\ref{eq:rate-cond2}) imply all the needed assumptions of
Theorem 30 from \cite{CZ-Meln} in the absence of the center coordinate.
\end{proof}

In above theorem we ignore (fix) the parameter. The result can be extended
to include the parameter as follows.

\begin{theorem}
\label{th:Wu-rates-theta}Consider a parameter dependent ODE 
\begin{equation*}
p^{\prime}=f(p,\theta),
\end{equation*}
for $\theta\in\Theta$. Assume that the system has a smooth family of
hypebolic fixed points $p_{\theta}^{\ast}$. Assume that for each (fixed) $%
\theta$, the vector field satisfies assumptions of Theorem \ref{th:Wu-rates}%
. Then the family of unstable manifolds $W_{\theta}^{u}$ (as defined in (\ref%
{eq:Wu-def})) of $p_{\theta}^{\ast}$ is given by a graph of a function%
\begin{equation*}
w^{u}:\overline{B}_{u}\left( R\right) \times\Theta\rightarrow\overline {B}%
_{s}\left( R\right) ,
\end{equation*}
(meaning that $W_{\theta}^{u}=\left\{ \left( x,w^{u}(x,\theta)\right) :x\in%
\overline{B}_{u}\left( R\right) \right\} $,) which is as smooth as $f$.
\end{theorem}

\begin{proof}
The existence of $w^{u}$ follows from Theorem \ref{th:Wu-rates}. We need to
justify its smoothness.

From the classical theory (see for instance \cite{Hartman},\cite{Hirsch},\cite{Nitecki}%
), we know that in a small neighbourhood $U$ of $\{(p_{\theta
}^{\ast },\theta )|\theta \in \Theta \}$ (considered in the state space,
extended to include the parameter), the family of local unstable manifolds
exists, and is as smooth as $f$. Condition (\ref{eq:contraction-cond})
ensures that the local manifold is propagated along the flow in the extended
space to span the set $D\times \Theta $. Since $\Phi _{t}$ is as smooth as $f
$, this establishes the smoothness of $w^{u}$.
\end{proof}

\begin{remark}
In this section we have focused on the unstable manifold. This method can
also be applied to obtain bounds on a stable manifold. To do so one can
simply change the sign of the vector field.
\end{remark}

\section{Dependence of the unstable manifold on parameters\label%
{sec:Wu-param}}

In this section we consider the ODE of the form 
\begin{equation}
p^{\prime}=f(p,\theta)  \label{eq:ode-theta}
\end{equation}
depending on the parameter $\theta\in\Theta$, where $p\in\mathbb{R}^{u}\times%
\mathbb{R}^{s}$ and $f:\mathbb{R}^{u}\times\mathbb{R}^{s}\times
\Theta\rightarrow\mathbb{R}^{u}\times\mathbb{R}^{s}$ is $C^1$
function, with 
\begin{equation*}
f(x,y,\theta)=(f_{x}(x,y,\theta),f_{y}(x,y,\theta)).
\end{equation*}
Our aim know is to examine the nature of the dependency of function $w^{u}$,
which parametrizes the unstable manifold in the stament of Theorem \ref%
{th:Wu-rates}, on parameter $\theta$.

Let our coordinates be $(x,y,\theta)\in\mathbb{R}^{u}\times\mathbb{R}%
^{s}\times\mathbb{R}$, and let us consider the following sets: 
\begin{align*}
J_{s}\left( q,M\right) & =\left\{ \left( x,y,\theta\right) :\left\Vert
\pi_{x,\theta}q-\left( x,\theta\right) \right\Vert \leq M\left\Vert \pi
_{y}q-y\right\Vert \right\} , \\
J_{cu}\left( q,M\right) & =\left\{ \left( x,y,\theta\right) :\left\Vert
\pi_{y}q-y\right\Vert \leq M\left\Vert \pi_{x,\theta}q-\left( x,\theta
\right) \right\Vert \right\} ,
\end{align*}
where $q\in\mathbb{R}^{u}\times\mathbb{R}^{s}\times\mathbb{R}$ and $M>0$.
These sets represent cones depicted in Figure \ref{fig:cones}. Note that we
have%
\begin{equation}
\left( J_{cu}\left( q,1/M\right) \right) ^{c}=\mathrm{int}J_{s}\left(
q,M\right) .  \label{eq:Jcu-complement}
\end{equation}

\begin{figure}[ptb]
\begin{center}
\includegraphics[height=4.5cm]{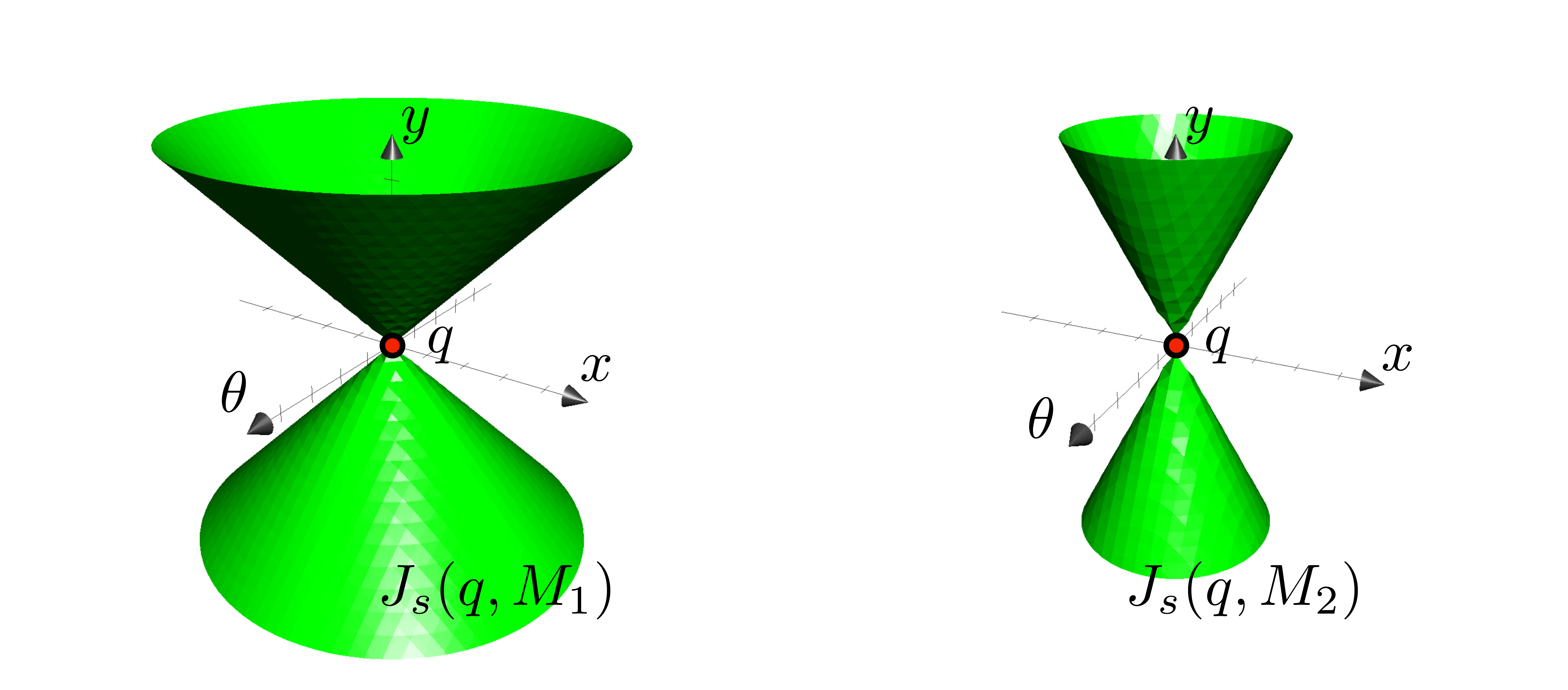}
\end{center}
\caption{The cones $J_s(q,M_1)$ and $J_s(q,M_2)$ for $M_1=1$ and $M_2=\frac{1%
}{2}$.}
\label{fig:cones}
\end{figure}

Let us consider an ODE given by (\ref{eq:ode-theta}) in the state space
extended by parameter, that is 
\begin{equation}
\left( x^{\prime},y^{\prime},\theta^{\prime}\right) =\left( f_{x}\left(
x,y,\theta\right) ,f_{y}\left( x,y,\theta\right) ,f_{\theta}(x,y,\theta)
\right) ,  \label{eq:ode-extended}
\end{equation}
where $f_{\theta}(x,y,\theta)=0$. Let $\Phi_{t}(x,y,\theta)$ be the flow
induced by (\ref{eq:ode-extended}).

Let $D=\overline{B}_{u}\left( R\right) \times\overline{B}_{s}\left( R\right)
\subset\mathbb{R}^{u}\times\mathbb{R}^{s}$ and let us define 
\begin{equation*}
\mathcal{D}=D\times\Theta,
\end{equation*}
and the following constants:%
\begin{align}
\mu\left( M\right) & =l\left( \frac{\partial f_{y}}{\partial y}\left( 
\mathcal{D}\right) \right) +M\left\Vert \frac{\partial f_{y}}{\partial
\left( x,\theta\right) }\left( \mathcal{D}\right) \right\Vert ,
\label{eq:mu-def} \\
\xi\left( M\right) & =m_{l}\left( \frac{\partial f_{x,\theta}}{%
\partial\left( x,\theta\right) }\left( \mathcal{D}\right) \right) -\frac{1}{M%
}\left\Vert \frac{\partial f_{x,\theta}}{\partial y}\left( \mathcal{D}%
\right) \right\Vert .  \label{eq:xi-def}
\end{align}
Our objective will be to prove the following theorem:

\begin{theorem}
\label{th:param-cone-prop}Consider that assumptions of Theorem \ref%
{th:Wu-rates-theta} hold and that $M>0$ is such that 
\begin{equation*}
\mu\left( M\right) <0\qquad\text{and\qquad}\xi\left( M\right) >\mu\left(
M\right) .
\end{equation*}
Then%
\begin{equation*}
\left\Vert \frac{\partial w^{u}}{\partial\theta}\right\Vert \leq1/M.
\end{equation*}
\end{theorem}

The proof of the theorem will be given at the end of the section. To show
the result we shall need two technical lemmas.

\begin{lemma}
\label{lem:param-cone-prop}Assume that $M>0$ is such that 
\begin{equation*}
\mu\left( M\right) <0\qquad\text{and\qquad}\xi\left( M\right) >\mu\left(
M\right) .
\end{equation*}
Then there exists a $c>0$ and $t_{M}>0$ such that for any $q\in\mathcal{D}$
and $p\in J_{s}\left( q,M\right) \cap\mathcal{D},$ $p\neq q$, as long as $%
\{\Phi_{-t}(p),\Phi_{-t}(q):t\in\lbrack0,t_{M}]\}\subset\mathcal{D}$, the
following inequality holds 
\begin{equation}
\left\Vert \pi_{y}\left( \Phi_{-t}\left( p\right) -\Phi_{-t}\left( q\right)
\right) \right\Vert >\left( 1+ct\right) \left\Vert \pi_{y}\left( p-q\right)
\right\Vert ,  \label{eq:Js-contraction}
\end{equation}
for any $t\in(0,t_{M})$. Moreover 
\begin{equation}
\Phi_{-t}\left( p\right) \in J_{s}\left( \Phi_{-t}\left( q\right) ,M\right) .
\label{eq:Js-alignment}
\end{equation}
\end{lemma}

\begin{proof}
Take any $q\in\mathcal{D}$ and $p\in J_{s}\left( q,M\right) \cap \mathcal{D}%
, $ $p\neq q$, and let $t>0$ be such that $\{\Phi_{-s}(p),\Phi
_{-s}(q):s\in\lbrack0,t]\}\subset\mathcal{D}$.

Since $p\in J_{s}\left( q,M\right) ,$ 
\begin{equation}
\left\Vert \pi_{x,\theta}\left( p-q\right) \right\Vert \leq M\left\Vert
\pi_{y}\left( p-q\right) \right\Vert .  \label{eq:tmp-cone-M-approx}
\end{equation}
As a consequence 
\begin{equation}
\left\Vert p-q\right\Vert \leq\sqrt{M^{2}+1}\left\Vert
\pi_{y}(p-q)\right\Vert .  \label{eq:tmp-p-q-norm-estim}
\end{equation}
Therefore since $p\neq q$ we must have 
\begin{equation*}
\left\Vert \pi_{y}(p-q)\right\Vert \neq0.
\end{equation*}

On the other hand, from Lemma \ref{lem:phi(p)-phi(q)} it follows that for
some $A\in\left[ \frac{\partial f_{y}}{\partial y}\left( \mathcal{D}\right) %
\right] $ and $B\in\left[ \frac{\partial f_{y}}{\partial\left(
x,\theta\right) }\left( \mathcal{D}\right) \right] $%
\begin{align*}
\pi_{y}\left( \Phi_{-t}\left( p\right) -\Phi_{-t}\left( q\right) \right) &
=\pi_{y}\left( p-q\right) -tA\pi_{y}\left( p-q\right) -tB\pi_{x,\theta
}\left( p-q\right) \\
& \quad+\pi_{y}g(t,p,q),
\end{align*}
where $g$ satisfies $\left\Vert g(t,p,q)\right\Vert
\leq\gamma_{1}t^{2}\left\Vert p-q\right\Vert $ for some constant $%
\gamma_{1}>0$. Observe that from (\ref{eq:tmp-p-q-norm-estim}) we have $%
\left\Vert g(t,p,q)\right\Vert \leq\gamma_{1}t^{2}\left\Vert
\pi_{y}(p-q)\right\Vert $. From the above, and by using (\ref%
{eq:tmp-cone-M-approx}) in the second line, Lemma \ref{lem:m-to-ml} in the
third line, Corollary \ref{cor:m-to-l} in the fourth line, and (\ref%
{eq:mu-def}) in the last line, we obtain 
\begin{align}
\left\Vert \pi_{y}\left( \Phi_{-t}\left( p\right) -\Phi_{-t}\left( q\right)
\right) \right\Vert & \geq\left\Vert \left( Id-tA\right) \pi_{y}\left(
p-q\right) \right\Vert -t\left\Vert B\right\Vert \left\Vert
\pi_{x,\theta}\left( p-q\right) \right\Vert  \notag \\
& \quad-\gamma_{1}t^{2}\left\Vert \pi_{y}(p-q)\right\Vert  \notag \\
& \geq\left( m\left( Id-tA\right) -tM\left\Vert B\right\Vert \right)
\left\Vert \pi_{y}\left( p-q\right) \right\Vert  \notag \\
& \quad-\gamma_{1}t^{2}\left\Vert \pi_{y}(p-q)\right\Vert  \notag \\
& \geq\left( 1+tm_{l}\left( -A\right) -tM\left\Vert B\right\Vert \right)
\left\Vert \pi_{y}\left( p-q\right) \right\Vert  \notag \\
& \quad-\gamma_{2}t^{2}\left\Vert \pi_{y}(p-q)\right\Vert  \notag \\
& =\left( 1+t\left( -l\left( A\right) -M\left\Vert B\right\Vert \right)
\right) \left\Vert \pi_{y}\left( p-q\right) \right\Vert  \notag \\
& \quad-\gamma_{2}t^{2}\left\Vert \pi_{y}(p-q)\right\Vert  \notag \\
& \geq\left( 1-t\mu\left( M\right) -\gamma_{2}t^{2}\right) \left\Vert
\pi_{y}\left( p-q\right) \right\Vert ,  \label{eq:y-projection}
\end{align}
where, in the light of Lemma \ref{lem:m-to-ml}, the third inequality is
satisfied for any $t\in\lbrack0,t_{0}]$, where $t_{0}>0$. Taking a fixed $%
c\in\left( 0,-\mu\left( M\right) \right) $, we see that there exists $%
t_{M}>0 $ (independent of $p$ and $q$) such that for any $t\in(0,t_{M})$ 
\begin{equation*}
\left\Vert \pi_{y}\left( \Phi_{-t}\left( p\right) -\Phi_{-t}\left( q\right)
\right) \right\Vert >\left( 1+tc\right) \left\Vert \pi_{y}\left( p-q\right)
\right\Vert ,
\end{equation*}
which proves (\ref{eq:Js-contraction}).

Again from Lemma \ref{lem:phi(p)-phi(q)}, we know that for some $A\in\left[ 
\frac{\partial f_{x,\theta}}{\partial\left( x,\theta\right) }\left( \mathcal{%
D}\right) \right] $ and $B\in\left[ \frac{\partial f_{x,\theta}}{\partial y}%
\left( \mathcal{D}\right) \right] $%
\begin{align*}
\pi_{x,\theta}\left( \Phi_{-t}\left( p\right) -\Phi_{-t}\left( q\right)
\right) & =\pi_{x,\theta}\left( p-q\right) -tA\pi_{x,\theta}\left(
p-q\right) -tB\pi_{y}\left( p-q\right) \\
& \quad+\pi_{x,\theta}g(t,p,q).
\end{align*}
Hence, using (\ref{eq:tmp-cone-M-approx}) in the second line, Lemma \ref%
{lem:norn-l-approx} in the third line, Corollary \ref{cor:m-to-l} in the
fourth line and (\ref{eq:xi-def}) in the last line,%
\begin{align}
\left\Vert \pi_{x,\theta}\left( \Phi_{-t}\left( p\right) -\Phi_{-t}\left(
q\right) \right) \right\Vert & \leq\left\Vert Id-tA\right\Vert \left\Vert
\pi_{x,\theta}\left( p-q\right) \right\Vert +t\left\Vert B\right\Vert
\left\Vert \pi_{y}\left( p-q\right) \right\Vert  \notag \\
& \quad+\gamma_{1}t^{2}\left\Vert \pi_{y}(p-q)\right\Vert  \notag \\
& \leq\left( \left\Vert Id-tA\right\Vert M+t\left\Vert B\right\Vert \right)
\left\Vert \pi_{y}\left( p-q\right) \right\Vert  \notag \\
& \quad+\gamma_{1}t^{2}\left\Vert \pi_{y}(p-q)\right\Vert  \notag \\
& =M\left( \left( 1+tl\left( -A\right) \right) M+\frac{1}{M}t\left\Vert
B\right\Vert \right) \left\Vert \pi_{y}\left( p-q\right) \right\Vert  \notag
\\
& \quad+\gamma_{2}t^{2}\left\Vert \pi_{y}(p-q)\right\Vert  \notag \\
& =M\left( \left( 1-tm_{l}\left( A\right) \right) +\frac{1}{M}t\left\Vert
B\right\Vert \right) \left\Vert \pi_{y}\left( p-q\right) \right\Vert  \notag
\\
& \quad+\gamma_{2}t^{2}\left\Vert \pi_{y}(p-q)\right\Vert  \notag \\
& \leq\left( M-tM\xi\left( M\right) +\gamma_{2}t^{2}\right) \left\Vert
\pi_{y}\left( p-q\right) \right\Vert ,  \label{eq:x-theta-projection}
\end{align}
where, in the light of Lemma \ref{lem:norn-l-approx}, the third inequality
is satisfied for any $t\in\lbrack0,t_{0}]$, where $t_{0}>0$. Since $%
\xi\left( M\right) >\mu\left( M\right) ,$ by combining (\ref{eq:y-projection}%
) with (\ref{eq:x-theta-projection}), we see that for sufficiently small $t$,%
\begin{equation*}
\frac{\left\Vert \pi_{x,\theta}\left( \Phi_{-t}\left( p\right) -\Phi
_{-t}\left( q\right) \right) \right\Vert }{\left\Vert \pi_{y}\left(
\Phi_{-t}\left( p\right) -\Phi_{-t}\left( q\right) \right) \right\Vert }\leq%
\frac{\left( M-tM\xi\left( M\right) +\gamma_{2}t^{2}\right) \left\Vert
\pi_{y}\left( p-q\right) \right\Vert }{\left( 1-t\mu\left( M\right)
-\gamma_{2}t^{2}\right) \left\Vert \pi_{y}\left( p-q\right) \right\Vert }%
\leq M.
\end{equation*}
This means that%
\begin{equation*}
\left\Vert \pi_{\theta,x}\left( \Phi_{-t}\left( p\right) -\Phi_{-t}\left(
q\right) \right) \right\Vert \leq M\left\Vert \pi_{y}\left( \Phi _{-t}\left(
p\right) -\Phi_{-t}\left( q\right) \right) \right\Vert ,
\end{equation*}
which proves (\ref{eq:Js-alignment}).
\end{proof}

We now return to studying (\ref{eq:ode-theta}). Let us assume that the
system has a smooth family of hyperbolic fixed points $%
p_{\theta}^{\ast}\in \mathrm{int}D$, where $D= \overline {B}_{u}\left(
R\right) \times\overline{B}_{s}\left( R\right) $.
 Let us also assume that for any given $\theta\in\Theta$
assumptions of Theorem \ref{th:Wu-rates} are satisfied. Let $w^{u}$ be the
parameterisation from Theorem \ref{th:Wu-rates-theta}.

\begin{lemma}
\label{lem:wu-in-cone}If assumptions of Theorem \ref{th:Wu-rates-theta} are
satisfied and 
\begin{equation*}
\mu\left( M\right) <0,\qquad\text{\qquad}\xi\left( M\right) >\mu\left(
M\right) ,
\end{equation*}
then for any $x_{1},x_{2}\in\overline{B}_{u}\left( R\right) $ and $%
\theta_{1},\theta_{2}\in\Theta$,%
\begin{equation}
\left( x_{1},w^{u}\left( x_{1},\theta_{1}\right) ,\theta_{1}\right) \in
J_{cu}\left( \left( x_{2},w^{u}\left( x_{2},\theta_{2}\right) ,\theta
_{2}\right) ,1/M\right) .  \label{eq:cu-jet-alignment}
\end{equation}
\end{lemma}

\begin{proof}
Let $q_{1}=\left( x_{1},w^{u}\left( x_{1},\theta_{1}\right) ,\theta
_{1}\right) $ and $q_{2}=\left( x_{2},w^{u}\left( x_{2},\theta_{2}\right)
,\theta_{2}\right) $. If (\ref{eq:cu-jet-alignment}) does not hold, then by (%
\ref{eq:Jcu-complement}) 
\begin{equation*}
q_{1}\in\mathrm{int}J_{s}\left( q_{2},M\right) .
\end{equation*}
Note that then%
\begin{equation*}
0\leq\left\Vert \pi_{x,\theta}\left( q_{1}-q_{2}\right) \right\Vert
<M\left\Vert \pi_{y}\left( q_{1}-q_{2}\right) \right\Vert .
\end{equation*}

By Lemma \ref{lem:param-cone-prop}, since $\Phi_{-t}\left( q_{i}\right) \in
W_{\theta_{i}}^{u}\times\{\theta_{i}\}\subset D\times\Theta$, 
we would therefore have%
\begin{equation}
\Phi_{-t}\left( q_{1}\right) \in J_{s}\left( \Phi_{-t}\left( q_{2}\right)
,M\right) ,  \label{eq:tmp-jts-align-big-t}
\end{equation}
for all $t\in\mathbb{R}_{+}$ (we can apply Lemma \ref{lem:param-cone-prop}
with small $t$ several times to obtain (\ref{eq:tmp-jts-align-big-t}) for
large $t$). Also by Lemma \ref{lem:param-cone-prop} we would have%
\begin{equation*}
\left\Vert \pi_{y}\left( \Phi_{-t}\left( q_{1}\right) -\Phi_{-t}\left(
q_{2}\right) \right) \right\Vert >\left( 1+ct\right) \left\Vert \pi
_{y}\left( q_{1}-q_{2}\right) \right\Vert \rightarrow\infty\qquad\text{as }%
t\rightarrow\infty.
\end{equation*}
This contradicts the fact that $\Phi_{-t}\left( p\right) ,\Phi_{-t}\left(
q\right) \in\mathcal{D}$, hence (\ref{eq:cu-jet-alignment}) must hold true.
\end{proof}

We are now ready to prove Theorem \ref{th:param-cone-prop}.

\begin{proof}[Proof of Theorem \protect\ref{th:param-cone-prop}]
By Theorem \ref{th:Wu-rates-theta}, $w^{u}$ is well defined. By Lemma \ref%
{lem:wu-in-cone},%
\begin{equation*}
\left\Vert w^{u}\left( x,\theta_{1}\right) -w^{u}\left( x,\theta _{2}\right)
\right\Vert \leq1/M\left\Vert \left( x,\theta_{1}\right) -\left(
x,\theta_{2}\right) \right\Vert =1/M\left\Vert \theta_{1}-\theta
_{2}\right\Vert ,
\end{equation*}
which implies the claim.
\end{proof}

\section{Computer assisted proof of the Shil'nikov connection in the Lorenz
84 system \label{sec:CAP}}

To apply our method and conduct a computer assisted proof we follow the
steps:

\begin{enumerate}
\item Using Theorem \ref{th:interval-Newton}, establish an enclosure of the
family of hyperbolic fixed points, and following the method from section \ref%
{sec:eigenvalues}, establish bounds on the eigenvalues of the Jacobian at
the fixed points to verify hyperbolicity.

\item \label{step:Wu} In local coordninates around the fixed points, using
Theorem \ref{th:Wu-rates-theta}, establish the bounds on the unstable
manifolds.

\item By changing sign of the vector field, using the same procedure as in
step \ref{step:Wu}, establish bounds on the stable manifolds.

\item Using Theorem \ref{th:param-cone-prop}, establish bounds on the
dependence of the manifolds on the parameter.

\item Propagate the bounds on the unstable manifold along the flow, and
establish the homoclinic intersection using Theorem \ref%
{th:homoclinic-existence}.
\end{enumerate}

For our computer assisted proof we consider the Lorenz 84 system (\ref%
{eq:L84}) with the parameters $a=\frac{1}{4}$,$b=4$, $F=4$, and 
\begin{equation}
G\in\left[ G_{l},G_{r}\right] =\left[ 0.0752761095,0.07527611625\right] .
\label{eq:G-again}
\end{equation}

We first use the interval Newton method (Theorem \ref{th:interval-Newton})
to establish an enclosure of the fixed points: 
\begin{equation*}
p_{G}^{\ast}\in\left( 
\begin{array}{l}
\lbrack3.9999144633,3.9999144654] \\ 
\lbrack-0.0008521960,-0.0008521939] \\ 
\lbrack0.0045450712,0.0045450733]%
\end{array}
\right) ,\qquad\text{for all }G\in\left[ G_{l},G_{r}\right] .
\end{equation*}
Next we compute a bound on the derivative of the vector field at the fixed
points, and using the method from section \ref{sec:eigenvalues} establish
that for all $G\in\left[ G_{l},G_{r}\right] $ the eigenvalues are:%
\begin{align*}
\lambda_{1} & \in\left[ 0.249988,0.249991\right] , \\
\mathrm{Re}\lambda_{2} & \in\left[ -2.999911,-2.999908\right] ,\qquad\mathrm{%
Im}\lambda_{2}\in\left[ 15.999657,15.999660\right] , \\
\mathrm{Re}\lambda_{3} & \in\left[ -2.999911,-2.999908\right] ,\qquad\mathrm{%
Im}\lambda_{3}\in\left[ -15.999660,-15.999657\right] .
\end{align*}
This establishes hyperbolicity.

To obtain bounds for the stable/unstable manifolds, we use the local
coordinates $(x,y_{1},y_{2})$, 
\begin{equation*}
\left( X,Y,Z\right) =C\left( x,y_{1},y_{2}\right) +q_{0},
\end{equation*}
with, 
\begin{align*}
q_{0} & =\left(
3.9999144643281,-0.00085219497131102,0.0045450722448356\right) , \\
C & =\left( 
\begin{array}{lll}
1 & -0.00016604653053618 & 0.00040407899883959 \\ 
0.00016384655297642 & -0.28235213046095 & 0.71764786953905 \\ 
-0.0011562746220118 & 0.71764798264861 & 0.28235189601999%
\end{array}
\right) .
\end{align*}

The $q_{0}$ is close to the fixed points of (\ref{eq:L84}). (Depending on
the choice of $G$ the fixed point shifts slightly with the parameter, but we
keep $q_{0}$ fixed.) Coordinates $x,y_{1},y_{2}$ align the system so that $x$
is the (rough) unstable direction, and $y_{1},y_{2}$ are (roughly) stable.

In these local coordinates, we use the interval Newton method (Theorem \ref%
{th:interval-Newton}) to obtain enclosures of the fixed points for
parameters $G$ in (\ref{eq:G-again}). In the local coordinates, the fixed
points are close to the origin. (See Figure \ref{Fig:Wu}; the cones emanate
from the fixed points.) We then choose 
\begin{equation*}
D=\overline{B}_{u}\left( R\right) \times\overline{B}_{s}\left( R\right) ,
\end{equation*}
with $R=10^{-4}$, and use Theorem \ref{th:Wu-rates-theta} to obtain an
enclosure of the unstable manifold $W^{u}$. In our computer assisted proof,
we have a Lipschitz bound $L_{u}=10^{-5}$ for the slope of the unstable
manifold for all parameters (\ref{eq:G-again}). See Figure \ref{Fig:Wu}.
(Note the scale on the axes. The enclosure is in fact quite sharp.)

\begin{figure}[ptb]
\begin{center}
\includegraphics[height=6cm]{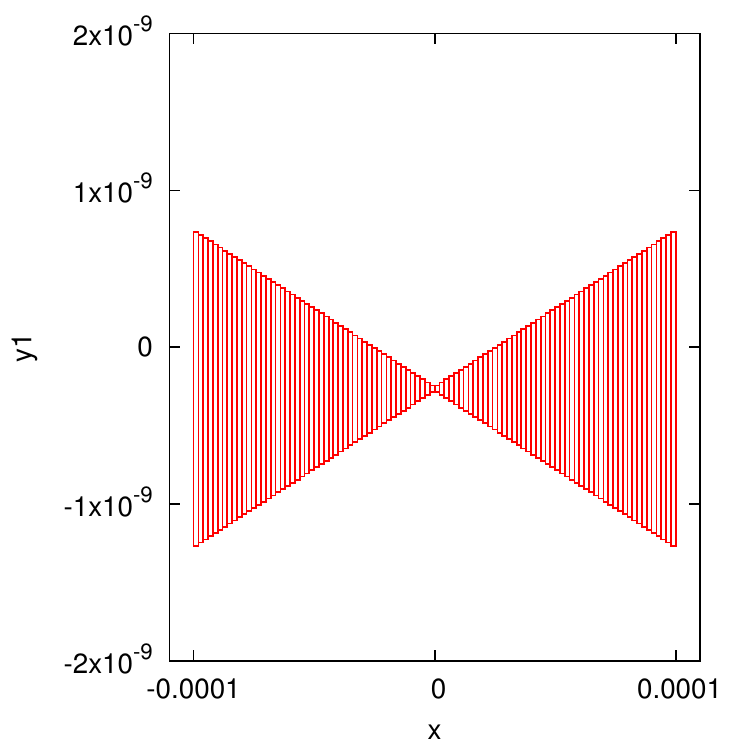}%
\includegraphics[height=6cm]{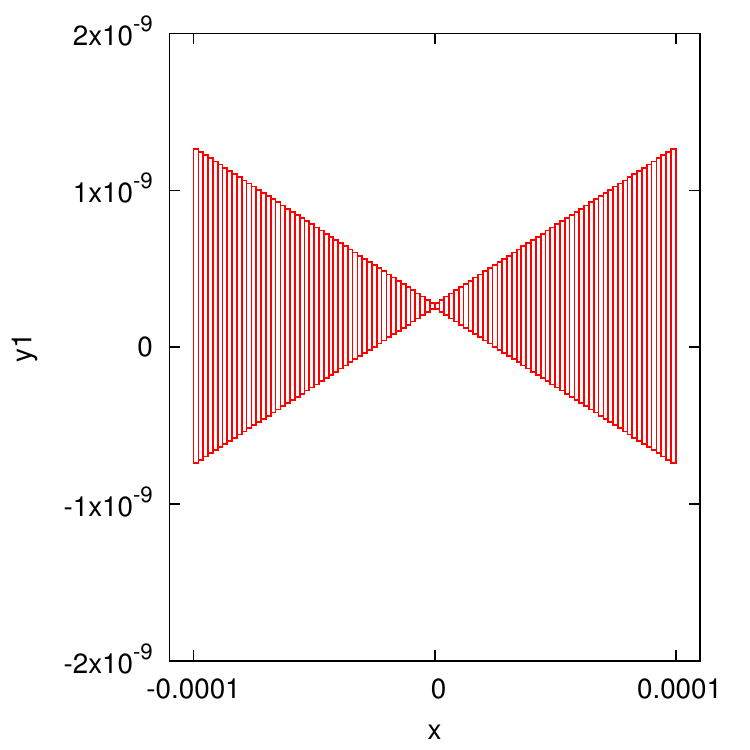}
\end{center}
\caption{The projection onto $x,y_{1}$ coordinates of the bounds on $%
W_{G}^{u}$. On the left we have the bound for $G=G_{l}$ (the left end of our
parameter interval (\protect\ref{eq:G-again})), and on the right for $%
G=G_{r} $.}
\label{Fig:Wu}
\end{figure}

To establish the bounds for the stable manifold, we consider the vector
field with reversed sign (which makes the stable manifold become unstable),
and apply Theorem \ref{th:Wu-rates-theta} once again. Here we have obtained
a Lipschitz bound $L_{s}=10^{-3}$. In Figure \ref{fig:Ws} we see the bound
on the enclosure. The two points on the plot are the $\Phi_{T}\left(
p_{G}^{u},G\right) $ for $G=G_{l}$ and $G=G_{r}$ for the choice of $T=50$
(see (\ref{eq:pu-theta-def}) for the definition of $p_{G}^{u}$). We do not
plot these as boxes, since our computer assisted bound gives their size of
order $10^{-13},$ and such boxes would be invisible on the plot. Note that
Figure \ref{fig:Ws} corresponds to the sketch from Figure \ref%
{fig:homoclinic-2}. In Figure \ref{fig:Ws} we have the projection onto $%
x,y_{1}$ coordinates of what happens inside of the set $D$, without plotting
the trajectory along the unstable manifold. 

We use the rigorous estimates for $\Phi_{T}\left( p_{G_{l}}^{u},G_{l}\right) 
$ and $\Phi_{T}\left( p_{G_{r}}^{u},G_{r}\right) $ to compute the following
bounds (see (\ref{eq:h-function-def}) for the definition of the function $h$%
,)%
\begin{align*}
h\left( G_{l}\right) & \in[1.193520892609e-07,1.2017042212622e-07], \\
h\left( G_{l}\right) & \in\left[ -1.1920396632516e-07,-1.1838527119022e-07%
\right] .
\end{align*}
We also make sure that $\Phi_{T}\left( p_{G}^{u},G\right) \in D$ for all $G%
\in[G_{l},G_{r}]$. We see that assumption (\ref{eq:Bolzano-assmpt}) of
Theorem \ref{th:homoclinic-existence} is satisfied, which means that we have
a Shil'nikov homoclinic connection for at least one of the parameters $G\in%
\left[ G_{l},G_{r}\right] $.

\begin{figure}[ptb]
\begin{center}
\includegraphics[height=6cm]{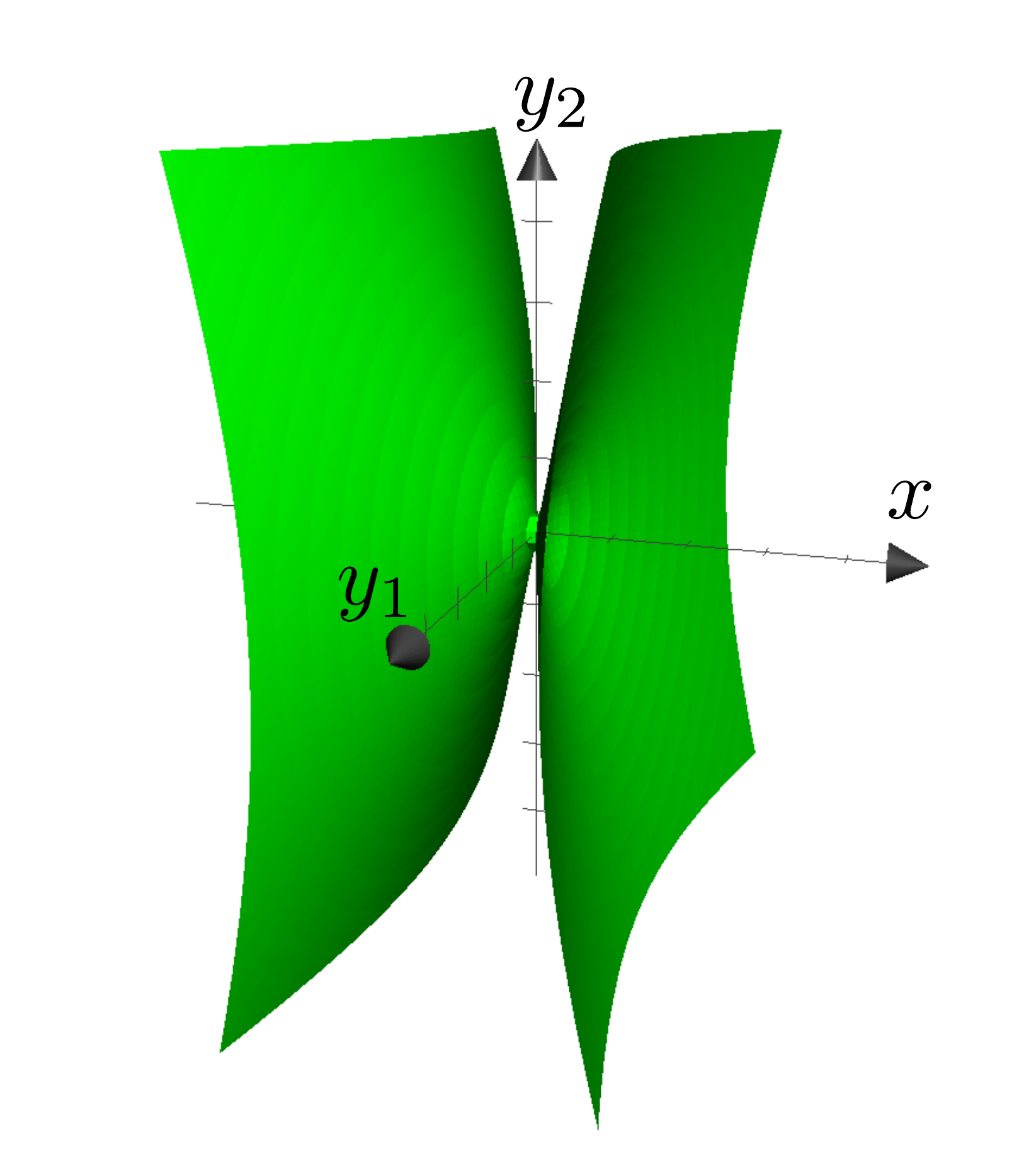}%
\includegraphics[height=6cm]{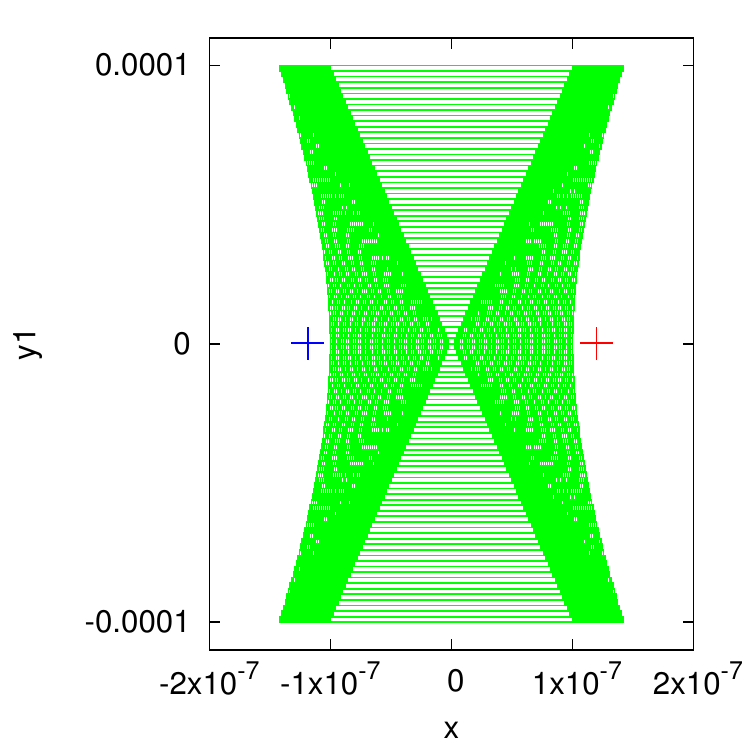}
\end{center}
\caption{The bound on $W_{G}^{s}$, for all parameters $G$ from (\protect\ref%
{eq:G-again}). On the left we have a non-rigorous plot, to illustrate the
shape of our bound in three dimensions. On the right, we have a projection
onto the $x,y_{1}$ coordinates of the rigorous, computer assisted enclosure.
The two points depicted on the right hand side plot are $\Phi_{T}\left(
p_{G_{l}}^{u},G_{l}\right) $ (on the right, in red) and $\Phi_{T}\left(
p_{G_{r}}^{u},G_{r}\right) $ (on the left, in blue). }
\label{fig:Ws}
\end{figure}

To establish the bound on $h^{\prime}\left( G\right) $, we first use Theorem %
\ref{th:param-cone-prop} to establish an estimate for $\frac{d}{dG}%
w^{u}\left( x,G\right) $. In our computer assisted proof we use Theorem \ref%
{th:param-cone-prop} with parameter $M=M_{u}:=2000$. We then use Theorem \ref%
{th:param-cone-prop} once again to establish bounds for $\frac{d}{dG}%
w^{s}\left( x,G\right) $. (Here, again, we reverse the sign of the vector
field to make the manifold unstable.) We establish the bound with $%
M_{s}=500. $ We then propagate the bound for $\frac{d}{dG}w^{u}\left(
x,G\right) $ using rigorous, computer assisted integration, to obtain the
bound%
\begin{equation*}
h^{\prime}\left( G\right) \in\left[ -36.12,-34.57\right] ,\qquad\text{for
all }G\in\left[ G_{l},G_{r}\right] .
\end{equation*}
This, by Theorem \ref{th:homoclinic-existence}, establishes the uniqueness
of the intersection parameter in $\left[ G_{l},G_{r}\right] $.

\begin{remark}
We do not rule out a possiblility that for some parameter $G\in\left[
G_{l},G_{r}\right] $ the trajectory $\Phi_{t}\left( p_{G}^{u},G\right) $
could exit $D$ and return again to intersect $W_{G}^{s}$. We have not done
such investigation, which would require a global consideration of the
system. What we establish is that we have a single parameter for which the
homoclinic orbit behaves as the one in Figure \ref{fig:homoclinic}.
\end{remark}

The computer assisted proof has been done entirely by using the CAPD%
\footnote{%
computer assisted proofs in dynamics: http://capd.ii.uj.edu.pl/} package and
took 4 seconds on a single core 3Ghz Intel i7 processor.

\section{Appendix\label{sec:appendix}}

\begin{proof}[Proof of Lemma \protect\ref{lem:phi(p)-phi(q)}]
Let us take any $t>0$ and any $p,q\in\mathbb{R}%
^{n}$ such that $\{\Phi_{-s}(p),\Phi_{-s}(q):s\in [ 0,t ]\}\subset U$.
Observe that since $U$ is convex, for any $u\in[0,1],$ \[Df\left( q+u\left(p-q\right) \right) \in [Df(U)].\] Moreover,
\begin{equation*}
f\left( p\right) -f\left( q\right) =\int_{0}^{1}Df\left( q+u\left(
p-q\right) \right) du\left( p-q\right) .
\end{equation*}
Using this we have 
\begin{align}
& \Phi_{-t}\left( p\right) -\Phi_{-t}\left( q\right) \notag \\
& =p-q-\int_{0}^{t}f\left( \Phi_{-s}\left( p\right) \right) -f\left(
\Phi_{-s}\left( q\right) \right) ds  \notag \\
& =p-q  \notag  \\
& \ -\int_{0}^{t}\int_{0}^{1}Df\left( \Phi_{-s}\left( q\right) +u\left(
\Phi_{-s}\left( p\right) 
 -\Phi_{-s}\left( q\right) \right) \right) du\left(
\Phi_{-s}\left( p\right) -\Phi_{-s}\left( q\right) \right) ds  \notag \\
& =p-q-\int_{0}^{t}C(s)\left( \Phi_{-s}\left( p\right) -\Phi_{-s}\left(
q\right) \right) ds,  \label{eq:app-expansion}
\end{align}
where $\{C(s)\}$ is a family of matrixes defined as%
\begin{equation*}
C(s)=\int_{0}^{1}Df\left( \Phi_{-s}\left( q\right) +u\left( \Phi_{-s}\left(
p\right) -\Phi_{-s}\left( q\right) \right) \right) du\in\left[ Df\left(
U\right) \right] .
\end{equation*}

Since $f$ is $C^{1}$ in $U$ and $U$ is compact, there exists a constant $L>0$
such that for any $p\in U$ 
\begin{equation*}
\left\Vert Df(p)\right\Vert \leq L.
\end{equation*}
Using standard Gronwall estimates gives that%
\begin{equation}
\Phi_{-s}\left( p\right) -\Phi_{-s}\left( q\right) =p-q+h\left( s,p,q\right)
,  \label{eq:app-Lip-bd}
\end{equation}
where $h$ satisfies 
\begin{equation*}
\left\Vert h\left( s,p,q\right) \right\Vert \leq\left( e^{s L}-1\right)
\left\Vert p-q\right\Vert .
\end{equation*}

We can return to (\ref{eq:app-expansion}) and substitute (\ref{eq:app-Lip-bd}%
) into the term under the integral to obtain 
\begin{align*}
\Phi_{-t}\left( p\right) -\Phi_{-t}\left( q\right) &
=p-q-\int_{0}^{t}C(s)\left( p-q+h\left( s,p,q\right) \right) ds \\
& =p-q-tC\left( p-q\right) +g\left( t,p,q\right) ,
\end{align*}
for%
\begin{equation*}
C:=\frac{1}{t}\int_{0}^{t}C(s)ds\in\left[ Df\left( U\right) \right] ,
\end{equation*}
and%
\begin{equation*}
g\left( t,p,q\right) :=-\int_{0}^{t}C(s)h\left( s,p,q\right) ds.
\end{equation*}
Observing that%
\begin{align*}
\left\Vert g\left( t,p,q\right) \right\Vert & \leq\max_{s\in\left[ 0,t %
\right] }\left\Vert C\left( s\right) \right\Vert \left\Vert p-q\right\Vert
\int_{0}^{t}\left( e^{s L}-1\right) ds \\
& =\max_{s\in\left[ 0,t \right] }\left\Vert C\left( s\right) \right\Vert
\left\Vert p-q\right\Vert \frac{1}{L}\left( e^{Lt }-Lt -1\right) \\
& \leq M\left\Vert p-q\right\Vert t^{2},
\end{align*}
gives the claim.
\end{proof}